\documentclass[12pt,twoside]{amsart}
\usepackage{amsthm,amsfonts,amssymb,amscd,euscript}

\usepackage[matrix,arrow]{xy}

\author{Florin Ambro} 
\address{Institute of Mathematics ``Simion Stoilow'' of the Romanian
Academy\\
P.O. BOX 1-764, RO-014700 Bucharest\\ 
Romania.}
\email{florin.ambro@imar.ro}

\newcommand{\isoto}{{\overset{\sim}{\rightarrow}}}

\newcommand{\C}{{\mathbb C}}
\newcommand{\Q}{{\mathbb Q}}
\newcommand{\Z}{{\mathbb Z}}
\newcommand{\N}{{\mathbb N}}
\newcommand{\R}{{\mathbb R}}

\newcommand{\cI}{{\mathcal I}}
\newcommand{\cM}{{\mathcal M}}

\newcommand{\cO}{{\mathcal O}}

\newcommand{\bH}{{\mathbb H}}  

\newcommand{\Char}{\operatorname{char}}

\newcommand{\emb}{\operatorname{emb}}

\newcommand{\Gr}{\operatorname{Gr}}

\newcommand{\Int}{\operatorname{int}}
\newcommand{\Ker}{\operatorname{Ker}}

\newcommand{\relint}{\operatorname{relint}}

\newcommand{\Sing}{\operatorname{Sing}}

\newcommand{\Spec}{\operatorname{Spec}}

\newcommand{\Supp}{\operatorname{Supp}}

\theoremstyle{plain}
\newtheorem{thm}{Theorem}[section]

\newtheorem{lem}[thm]{Lemma}
\newtheorem{cor}[thm]{Corollary}

\newtheorem{prop}[thm]{Proposition}

\theoremstyle{definition}

\newtheorem{exmp}[thm]{Example}

\newtheorem{rem}[thm]{Remark}

\newtheorem{ack}{Acknowledgments}   

\theoremstyle{remark}

\begin{document}

\bibliographystyle{amsalpha+}
\title{On toric face rings I}
\maketitle

\begin{abstract} 
We construct an explicit Deligne - Du Bois complex for algebraic varieties which are 
locally analytically isomorphic to the spectrum of a toric face ring.
\end{abstract} 





\section*{Introduction}


Our motivation to study toric face rings is to construct toric examples of semi-log canonical
singularities (cf.~\cite{Kbook}).
It is known that for the class of log canonical singularities, (normal) toric examples form a reasonably 
large subclass, useful for testing open problems. These models can be defined either geometrically 
as normal affine equivariant torus embeddings $T\subset X$,  or algebraically as 
$X=\Spec \C[M\cap \sigma]$, where $M$ is lattice and $\sigma\subset M_\R$ is a rational polyhedral 
cone. Here $\C[M\cap \sigma]=\oplus_{m\in M\cap \sigma}\C \cdot \chi^m$ is a semigroup ring
with multiplication $\chi^m\cdot \chi^{m'}=\chi^{m+m'}$.

 Semi-log canonical singularities appear at the boundary of the moduli space of canonically polarized
varieties with log canonical singularities. These singularities are weakly normal, but not necessarily
normal or even irreducible. Here are two examples:
\begin{itemize}
\item The {\em pinch point} is the surface singularity with local analytic model $$0\in X:(xy^2-z^2=0)\subset \C^3.$$
We have $X=\Spec \C [S]$, where $\C [S]=\oplus_{s\in S}\C \cdot \chi^s$ is the semigroup algebra 
associated to the semigroup $S=\N^2_{x_2>0}\sqcup 2\N\times 0$. The multiplication is given by 
$\chi^s\cdot \chi^{s'}=\chi^{s+s'}$. The torus $T=\Spec \C[\Z^2]$ acts naturally on $X$,
and $T\subset X$ becomes an affine equivariant torus embeddings, which is irreducible but not normal.
\item The {\em normal crossings singularity} has the local analytic model 
$$
0\in X:(\prod_{i=1}^q z_i=0)\subset \C^{d+1} \ (1\le q\le d+1) .
$$
The torus $T=\Spec \C[\Z^{d+1}]$ acts naturally on $\C^{d+1}$ and leaves $X$ invariant. In fact, 
$T$ acts on each irreducible component of $X$, inducing a structure of equivariant embedding of a 
torus which is a quotient of $T$. Corresponding to this action, we can write $X=\Spec \C [\cup_{i=1}^q S_i]$,
where $S_i=\{s\in \N^{d+1};s_i=0\}$ and $\C [\cup_{i=1}^q S_i]$ is the Stanley-Reisner ring with $\C$-vector space structure
$\oplus_{s\in \cup_{i=1}^qS_i}\C\cdot \chi^s$, and multiplication defined as follows: $\chi^s\cdot \chi^{s'}$ is
$\chi^{s+s'}$ if there exists $i$ such that $s,s'\in S_i$, and zero otherwise. 
\end{itemize}

Toric face rings are a natural generalization of semigroup rings and Stanley-Reisner rings.
We will use the definition of Ichim and R\"omer~\cite{IR07}, which is based on previous work of Stanley,
Reisner, Bruns and others (see the introduction of~\cite{IR07}). 
A {\em toric face ring} $\C[\cM]$ is associated to a {\em monoidal complex} $\cM=(M,\Delta,(S_\sigma)_{\sigma\in \Delta})$,
the data consisting of a lattice $M$, a fan $\Delta$ consisting of rational polyhedral cones in $M$, and a collection of semigroups 
$S_\sigma\subseteq M\cap \sigma$, such that $S_\sigma$ generates the cone $\sigma$ and $S_\tau=S_\sigma\cap \tau$
if $\tau$ is a face of $\sigma$. The $\C$-vector space structure is 
$$
\C[\cM]=\oplus_{s\in \cup_{\sigma\in \Delta} S_\sigma}\C\cdot \chi^s,
$$ 
with the following multiplication: $\chi^s\cdot \chi^{s'}$ is
$\chi^{s+s'}$ if there exists $\sigma\in \Delta$ such that $s,s'\in S_\sigma$, and zero otherwise. 

Toric face rings are glueings of semigroup rings, as $\C[\cM] \simeq \varprojlim_{\sigma\in \Delta}\C[S_\sigma]$.
The affine variety $X=\Spec \C[\cM]$ has a natural action by the torus $T=\Spec \C[M]$, and the cones of the fan
are in one to one correspondence with the orbits of the action. The $T$-invariant closed subvarieties of $X$ are also induced
by a toric face ring, obtained by restricting the fan $\Delta$ to a subfan. 

 We say that $X$ has {\em weakly toroidal singularities} if $X$ is weakly normal, and locally analytically isomorphic
to $\Spec \C[\cM]$ for some monoidal complex $\cM$. In the sequel to this paper~\cite{Am17}, we determine when 
$X$ has semi-log canonical singularities. In this paper, we aim to understand the topology of $X$.

Let $X/\C$ be a proper variety. If $X$ is smooth, the cohomology of $X^{an}$ is determined by differential forms 
on $X$~\cite{Del71}: the filtered complex $(\Omega^*_X,F)$, where $\Omega^*_X$ is the de Rham complex of 
regular differential forms on $X$, and $F$ is the naive filtration, induces in hypercohomology a
spectral sequence 
$$
E_1^{pq}=H^q(X,\Omega^p_X)\Longrightarrow \Gr_F^p H^{p+q}(X^{an},\C)
$$
which degenerates at $E_1$, and converges to the Hodge filtration on the cohomology groups of $X^{an}$.
If $X$ has singularities, its topology is determined by rational forms on a smooth simplicial resolution~\cite{Del74, DB81}:
if $\epsilon\colon X_\bullet\to X$ is a smooth simplicial resolution, the Deligne-Du Bois filtered complex
$$
(\underline{\Omega}^*_X,F):=R\epsilon_*(\Omega^*_{X_\bullet},F)
$$
induces in hypercohomology a spectral sequence 
$$
E_1^{pq}=\bH^q(X,\Gr^p_F\underline{\Omega}^p_X[p])\Longrightarrow \Gr_F^p H^{p+q}(X^{an},\C)
$$
which degenerates at $E_1$, and converges to the Hodge filtration on the cohomology groups of $X^{an}$.
The filtered complex $(\underline{\Omega}^*_X,F)$ does not depend on the choice of $\epsilon$, when viewed 
in the derived category of filtered complexes on $X$. It is a rather complicated object in general: for example $F$
may not be a naive filtration, so each $\Gr^p_F\underline{\Omega}^*_X[p]$ is a complex.

Steenbrink, Danilov, Du Bois and Ishida have observed
that if the singularities of $X$ are simple enough, one can still compute the cohomology of $X$ using differential
forms on $X$:
\begin{itemize}
\item Suppose $X$ has only quotient singularities, or toroidal singularities (i.e. locally analytically isomorphic to 
a normal affine toric variety). Let $w\colon U\subset X$ be the smooth locus of $X$. The complement has 
codimension at least $2$, since $X$ is normal. In particular, 
$$
\tilde{\Omega}^p_X:=w_*(\Omega^p_U)
$$ 
is a coherent $\cO_X$-module. Then $(\tilde{\Omega}^*_X,F_{naive})$ is a canonical (functorial) choice for the 
Deligne-Du Bois complex. In particular, the cohomology of $X$ is determined by rational differential forms on $X$
which are regular on the smooth locus of $X$~\cite{Stee77,Dan78,Dan91}. 
\item Suppose $X$ has normal crossings singularities. Let $\epsilon_0\colon X_0\to X$ be the normalization of $X$,
let $X_1=X_0\times_X X_0$. Both $X_0$ and $X_1$ are smooth, and if we define 
$$
\tilde{\Omega}^p_X:=\Ker(\epsilon_0\Omega^p_{X_0}\rightrightarrows \epsilon_1\Omega^p_{X_1}),
$$
then $(\tilde{\Omega}^*_X,F_{naive})$ is a canonical (functorial) choice for the 
Deligne-Du Bois complex~\cite{DB81}.
\item Let $Y=\Spec \C[M\cap \sigma]$ be a normal affine toric variety. Let $X\subset Y$ be a $T$-invariant
closed subvariety. One can define combinatorially a coherent $\cO_X$-module $\tilde{\Omega}^p_X$ (a glueing 
of certain regular forms on the orbits of $X$), such that $(\tilde{\Omega}^*_X,F_{naive})$ is a canonical (functorial) 
choice for the Deligne-Du Bois complex. The same holds for a semi-toroidal variety with a good 
filtration~\cite{Ish85}.
\end{itemize}

The aim of this paper is to unify all these results, and extend them to varieties with weakly toroidal singularities.
What all these examples have in common is the {\em vanishing property}
$$
R^i\epsilon_*\Omega^p_{X_\bullet}=0 \ (i> 0),
$$
where $\epsilon\colon X_\bullet\to X$ is a smooth simplicial resolution. This means that in the filtered derived 
category, the Deligne-Du Bois complex is equivalent to $(\tilde{\Omega}^*_X,F_{naive})$, where
$$
\tilde{\Omega}^p_X:=h^0(\underline{\Omega}^p_X)=\epsilon_*(\Omega^p_{X_\bullet})=\Ker(\epsilon_0\Omega^p_{X_0}\rightrightarrows \epsilon_1\Omega^p_{X_1})
$$
is the cohomology in degree zero of the complex $\underline{\Omega}^p_X$.
As defined, $\tilde{\Omega}^p_X$ is uniquely defined only up to an isomorphism. But if we require that 
$\epsilon_0\colon X_0\to X$ is a desingularization, then $\tilde{\Omega}^p_X$ is uniquely defined, and 
has a description in terms of rational differential forms on $X$. More precisely, let $\epsilon_0\colon X_0\to X$
be a desingularization, let $X_1\to X_0\times_X X_0$ be a desingularization. Then $\tilde{\Omega}^p_X$
consists of rational differential $p$-forms $\omega$ on $X$ such that $\omega$ is regular on the smooth
locus of $X$, the rational differential $\epsilon_0^*\omega$ extends to a regular $p$-form everywhere on $X_0$,
and $p_1^*\epsilon_0^*\omega=p_2^*\epsilon_0^*\omega$ on $X_1$. The $\cO_X$-module $\tilde{\Omega}^p_X$
coincides with the {\em sheaf of $h$-differential forms} $(\Omega^p_h)|_X$ introduced by Huber and J\"order~\cite{HJ14}.
It is functorial in $X$.

The main result of this paper is that weakly toroidal singularities satisfy the same vanishing property, hence 
the sheaf of $h$-differentials, endowed with the naive filtration, computes the cohomology of $X^{an}$:

\begin{thm}\label{mh} Let $X/\C$ be a variety with weakly toroidal singularities.
\begin{itemize}
\item[a)] The filtered complex $(\tilde{\Omega}^*_X,F_{naive})$, consisting of the sheaf of $h$-differential forms
on $X$ and its naive filtration, is a canonical (and functorial) choice for the Deligne-Du Bois complex of $X$.
\item[b)] $X$ has Du Bois singularities (i.e. $\cO_X=\tilde{\Omega}^0_X$).
\item[c)] Moreover, suppose $X/\C$ is proper.  Then $(\tilde{\Omega}^*_X,F_{naive})$ 
induces in hypercohomology a spectral sequence 
$$
E_1^{pq}=H^q(X,\tilde{\Omega}^p_X)\Longrightarrow \Gr_F^p H^{p+q}(X^{an},\C)
$$
which degenerates at $E_1$, and converges to the Hodge filtration on the cohomology groups of $X^{an}$.
\end{itemize}
\end{thm}

A similar result holds for pairs $(X,Y)$ with weakly toroidal singularities. In Theorem~\ref{mh}.c), we can say nothing
about the weight filtration on $H^*(X^{an},\C)$. 

We outline the structure of this paper. We recall in Section 1 the main result of Du Bois~\cite{DB81}, defining
from this point of view the sheaf of $h$-differentials of Huber and J\"order~\cite{HJ14}, and recall
the combinatorial description of differential forms on smooth toric varieties (used in Section 3).
Section 2 brings together mostly known results on affine equivariant embeddings of the torus, and toric
face rings. Especially, we see the combinatorial construction of weak (semi-) normalization of a toric
face ring. In Section 3 we give a combinatorial description for the sheaf of $h$-differentials on the spectrum
of a toric face ring, and prove the main vanishing result (Theorem~\ref{mv}). The proof is by induction on 
dimension; it is simpler but inspired from the proof of similar results of Danilov and Ishida. We also extend 
Theorem~\ref{mv} to toric pairs. In Section 4 we generalize the results of Section 3 to varieties with weakly 
toroidal singularities (pairs as well).

\begin{ack} I would like to thank Philippe Gille, Nguyen Dang Hop and Bogdan Ichim  for useful discussions.
\end{ack}


\section{Preliminary}



\subsection{Simplicial resolutions~\cite{Del74,DB81}}

See~\cite{Del74} for the definition of simplicial schemes.
Let $X/k$ be a scheme of finite type over a field of characteristic zero.
A {\em resolution of $X$} is an augmented simplicial $k$-scheme $\epsilon \colon X_\bullet\to X$
such that 
\begin{itemize}
\item the transition maps of $X_\bullet$ and the $\epsilon_n$'s are all proper, and
\item $(\Q_l)_X \to R\epsilon_*((\Q_l)_{X_\bullet})$ is an isomorphism (\'etale topology).
\end{itemize}
The resolution is called {\em smooth} if the components of $X_\bullet$ are smooth.

\begin{lem}~\cite[2.1.4, 2.4]{DB81}\label{sh}
Let $X'_\bullet\to X$ and $X_\bullet '' \to X$ be two (resp. smooth) resolutions. Then there exists a commutative diagram
\[ \xymatrix{
    &  X_\bullet \ar[dl] \ar[dr] &    \\
    X_\bullet  ' \ar[dr]  &   &  X_\bullet ''  \ar[dl]   \\
       & X  &
} \]
such that the composition $X_\bullet\to X$ is a (resp. smooth) resolution.
\end{lem}

\begin{thm}~\cite[3.11, 3.17,4.2]{DB81}\label{mDB}
Consider a commutative diagram 
\[ \xymatrix{
   X'_\bullet \ar[rr]^\alpha \ar[dr]_{\epsilon'}  & & X_\bullet \ar[dl]^\epsilon \\
               & X  & 
} \]
where $\epsilon,\epsilon'$ are smooth resolutions. Then 
$R\epsilon_* (\Omega^p_{X_\bullet}\to \alpha_*\Omega^p_{X'_\bullet}\to R\alpha_*\Omega^p_{X'_\bullet})$
induces a quasi-isomorphism $$R\epsilon_*(\Omega^p_{X_\bullet})\to R\epsilon'_*(\Omega^p_{X'_\bullet}).$$
\end{thm}

Taking cohomology in degree zero, we obtain that $\epsilon_*(\Omega^p_{X_\bullet})\to \epsilon'_*(\Omega^p_{X'_\bullet})$
is an isomorphism.

\begin{cor}\label{sno}
Let $\epsilon\colon X_\bullet\to X$ be a smooth resolution. If $X$ is smooth, the natural homomorphism 
$$
\epsilon^*\colon \Omega^p_X\to R\epsilon_*(\Omega^p_{X_\bullet})
$$ 
is a quasi-isomorphism. That is 
$\Omega^p_X\isoto \epsilon_*(\Omega^p_{X_\bullet})$ and $R^i\epsilon_*(\Omega^p_{X_\bullet})=0\ (i> 0)$.
\end{cor}

\begin{proof}
Factor $\epsilon$ through the constant resolution of $X$.
\end{proof}


\subsection{$h$-Differentials~\cite{HJ14}}

Let $X/k$ be a scheme of finite type, defined over a field of characteristic zero. 
Let $w\colon X^o\subseteq X$ be the smooth locus of $X/k$. Let $\epsilon\colon X_\bullet\to X$
be a smooth simplicial resolution. By Corollary~\ref{sno}, 
$\Omega^p_{X/k} \to \epsilon_*(\Omega^p_{X_\bullet})$ is an isomorphism over $X^o$.
Define a coherent $\cO_X$-module $\tilde{\Omega}^p_{X/k}$ as follows:
if $U\subseteq X$ is an open subset, $\Gamma(U,\tilde{\Omega}^p_{X/k})$ consists of those
differential forms $\omega\in \Gamma(U\cap X^o,\Omega^p_{U\cap X^0})$ such that 
$\epsilon^*\omega\in \Gamma(\epsilon^{-1}(U\cap X^o),\Omega^p_{X_\bullet})$ extends to a 
section of $\Gamma(\epsilon^{-1}(U),\Omega^p_{X_\bullet})$.

By Lemma~\ref{sh} and Theorem~\ref{mDB}, the definition of $\Omega^p_{X/k}$ does not depend
on the choice of $\epsilon$. Moreover, for every smooth simplicial resolution $\epsilon\colon X_\bullet\to X$,
we have an isomorphism 
$$
\epsilon_0^*\colon \tilde{\Omega}^p_X \isoto \epsilon_*(\Omega^p_{X_\bullet}).
$$ 
We have an induced $k$-linear differential $d\colon \tilde{\Omega}^p_X\to \tilde{\Omega}^{p+1}_X$, which 
defines a differential complex $\tilde{\Omega}^*_X$.

The correspondence $X\mapsto \tilde{\Omega}^*_X$ is functorial. Indeed, let $f\colon X'\to X$ be a morphism. 
There exists a commutative diagram 
\[ 
\xymatrix{
X'_\bullet   \ar[r]^{f_\bullet} \ar[d]_{\epsilon'} &  X_\bullet  \ar[d]^\epsilon    \\
X'   \ar[r]_f       &  X   
} \]
where $\epsilon$ and $\epsilon'$ are smooth simplicial resolutions.
The natural homomorphism $f_\bullet^*\colon \Omega^p_{X_\bullet/k} \to {f_\bullet}_*\Omega^p_{X'_\bullet/k}$ pushes forward
to $f^*\colon \tilde{\Omega}^p_{X/k}\to f_*\tilde{\Omega}^p_{X'/k}$. The latter does not depend on the choice of $\epsilon,\epsilon'$ 
and $f_\bullet$. Transitivity follows from this.

The natural homomorphism $\Omega^p_{X/k} \to \tilde{\Omega}^p_{X/k}$ is an isomorphism
over the smooth locus of $X$. 

The sheaf $\tilde{\Omega}^p_{X/k}$ coincides with the {\em sheaf of $h$-differential forms} $(\Omega^p_h)|_X$ introduced in~\cite{HJ14}.


\subsection{Differential forms on smooth toric varieties~\cite{Dan78}}


- Let $M$ be a lattice, let $\sigma\subset M_\R$ be a cone generated by finitely many elements of $M$.
Then $M\cap \sigma-M\cap \sigma=M\cap (\sigma-\sigma)$.

- Let $T=\Spec k[M]$ be a torus. Then $\Gamma(T,\Omega^p_T)=\oplus_{m\in M}\chi^m\cdot \wedge^p V$,
where $V$ is the $k$-vector space of $T$-invariant global $1$-forms on $T$. We have an isomorphism
$$
\alpha\colon k\otimes_\Z M\isoto V, 1\otimes m\mapsto \frac{d(\chi^m)}{\chi^m}.
$$
Moreover, $d\omega=0$ for every $\omega\in V$.

- Let $X=T_N\emb(\Delta)$ be a smooth torus embedding. Suppose $\Supp \Delta=\sigma^\vee$,
where $\sigma\subset M_\R$ is a rationally polyhedral cone. Then 
$
\Gamma(X,\Omega^p_X)=\oplus_{\tau\prec \sigma}\oplus_{m\in M\cap \relint \tau}\chi^m\cdot \wedge^p \alpha(M\cap \tau-M\cap \tau).
$
\begin{proof}
The restriction $\Gamma(X,\Omega^p_X)\to \Gamma(T,\Omega^p_T)$ is injective. Every element of the right hand side 
has a unique decomposition $\omega=\sum_{m\in M}\chi^m\omega_m$, with $\omega_m\in \wedge^pV$. It remains to 
identify which $\omega$ lift to $X$. Each $\omega$ extends as a form on $X$ with at most logarithmic poles along $X\setminus T$.
Then $\omega$ lifts to a regular form on $X$ if and only if $\omega$ is regular at the generic point of $V(e)$,
for every invariant prime $V(e)\subset X$, if and only if $e\in \Delta(1)$ and $\omega_m\ne 0$ implies 
$\langle m,e\rangle \ge 0$, and $\langle m,e\rangle =0$ implies $\omega_m\in \wedge^p \alpha(M\cap e^\perp)$.
This gives the claim.
\end{proof}

For each $m\in M\cap \sigma$, denote by $\sigma_m$ the unique face of $\sigma$ which contains $m$ in its
relative interior. Denote by $V_m$ the invariant regular $1$-forms
on the torus $\Spec k[M\cap \sigma_m-M\cap \sigma_m]$. Then we can rewrite
$$
\Gamma(X,\Omega^p_X)=\oplus_{m\in M\cap \sigma}\chi^m\cdot \wedge^pV_m.
$$


\section{Toric face rings}


All varieties considered are reduced schemes of finite type, defined over an algebraically closed field $k$, 
of characteristic $p\ge 0$.

Let $M$ be a lattice. It induces a $k$-algebra 
$k[M]=\oplus_{m\in M}k\cdot \chi^m$, with multiplication $\chi^m\cdot \chi^{m'}=\chi^{m+m'}$.
The variety $T=\Spec k[M]$ is called a {\em torus over $k$}. It is endowed a natural multiplication $T\times T\to T$,
given by translation on $M$.


\subsection{Equivariant affine embeddings of torus~\cite{KKMS73}}


Let $S\subseteq M$ be a finitely generated semigroup such that $S-S=M$. It induces a $k$-algebra 
$k[S]=\oplus_{m\in S}k\cdot \chi^m$, with the multiplication $\chi^m\cdot \chi^{m'}=\chi^{m+m'}$.
The affine variety $X=\Spec k[S]$ is an {\em equivariant embedding of $T$}: it is equipped with
a torus action $T\times X\to X$, and $X$ admits an open dense orbit isomorphic to $T$, such that the 
restriction of the action to this orbit corresponds to the torus multiplication.

The correspondence $S\mapsto \Spec k[S]$ is a bijection between finitely generated 
semigroups $S\subseteq M$ such that $S-S=M$, and isomorphism classes of affine equivariant
embeddings of $T$~\cite[Proposition 1]{KKMS73}. The semigroup is recovered as the set of
exponents of the torus action.

For the rest of this section, let $X=\Spec k[S]$. It is affine, reduced and irreducible. The torus orbits are in one to 
one correspondence with the faces of the cone $\sigma_S\subseteq M_\R$ generated by $S$.
If $\sigma$ is a face of $\sigma_S$, the ideal $k[S\setminus \sigma]$ defines a $T$-invariant closed
irreducible subvariety $X_\sigma \subseteq X$. We have $\tau\prec \sigma$ if and only if $X_\tau\subseteq X_\sigma$.
The orbit corresponding to the face $\sigma\prec \sigma_S$ is 
$O_\sigma=X_\sigma \setminus \cup_{\tau\prec \sigma, \tau\ne \sigma }X_\tau$, and is isomorphic to 
the torus $\Spec k[S\cap \sigma-S\cap \sigma]$.

The normalization of $X$ is
$
\bar{X}=\Spec k[\bar{S}] \to \Spec k[S]=X,
$
where $\bar{S}=\cup_{n\ge 1} \{m\in M;nm\in S\}=M\cap \sigma_S$ (see~\cite[Chapter 1]{KKMS73} for proofs of
the above statements). 

Recall~\cite{Tr70} that the {\em seminormalization of $X$}, denoted $X^{sn}\to X$, is defined as a universal homeomorphism 
$f\colon Y \to X$ such that $k(f(y))\to k(y)$ is an isomorphism for all Grothendieck points $y\in Y$, and $f$ is maximal with
this property. It follows that $X^{sn}\to X$ is birational, and topologically a homeomorphism.
We call $X$ {\em seminormal} if its seminormalization is an isomorphism.

\begin{prop}\cite[Proposition 5.32]{HR76} The seminormalization of $X$ is $\Spec k[S^{sn}] \to \Spec k[S]$, where 
$$
S^{sn}=\sqcup_{\sigma \prec \sigma_S} (S\cap \sigma-S\cap \sigma)\cap \relint \sigma.
$$
\end{prop}

\begin{proof} The seminormalization is the spectrum of the ring
$$
R=\cap_{x\in X} \{f\in k[\bar{S}]; f_x\in \cO_{X,x}+\text{Rad} (\pi_*\cO_{\bar{X}})_x \}.
$$
The normalization is a toric morphism, hence $R$ is $T$-invariant. Therefore 
$R=k[S^{sn}]$ for a certain semigroup $S\subseteq S^{sn}\subseteq \bar{S}$ which 
we identify. 

Let $m\in S^{sn}$. Let $x$ be the generic point of $X_\sigma$, for a face $\sigma\prec \sigma_S$.
There is a unique (invariant) point $x'$ lying over $x$, which is $\bar{X}_\sigma$.
The map $x'\to x$ corresponds to the morphism of tori
$$
O_{x'}=\Spec(k[\bar{S}\cap \sigma -\bar{S}\cap \sigma]) \to \Spec(k[S\cap \sigma -S\cap \sigma])=O_x.
$$
Now $\chi^m|_{O_{x'}}$ is $\chi^m$ if $m\in \sigma$, and $0$ otherwise. So the condition over
$x$ is that if the face $\sigma\prec \sigma_S$ contains $m$, then $m\in S\cap \sigma -S\cap \sigma$.
The condition for $\chi^m$ over all torus invariant points of $X$ is thus equivalent to:
if $\sigma$ is the unique face of $\sigma_S$ which contains $m$ in its relative interior, then $m\in S\cap \sigma-S\cap \sigma$.
That is $m$ belongs to 
$$
S'=\sqcup_{\sigma\prec \sigma_S} (S\cap \sigma-S\cap \sigma)\cap \relint \sigma.
$$
To check that $\chi^m$ satisfies the gluing condition over all points of $X$, it suffices now to
show that $\pi'\colon \Spec k[S']\to \Spec k[S]$ is a homeomorphism, which induces 
isomorphism between residue fields. Our map respects the orbit decompositions 
$\Spec k[S']=\sqcup_\sigma O'_\sigma \to \Spec k[S]=\sqcup_\sigma O_\sigma$, 
and $O'_\sigma \to O_\sigma$
is isomorphic to $\Spec k[S'\cap \sigma-S'\cap \sigma] \to \Spec k[S\cap \sigma-S\cap \sigma]$. The latter is an
isomorphism since $S\cap \sigma-S\cap \sigma=S'\cap \sigma-S'\cap \sigma$. We deduce that $\pi'$ is bijective.
It is also proper, hence open. Therefore $\pi'$ is a homeomorphism. Since the maps
between orbits are isomorphisms, and the orbits are locally closed, it follows that $\pi'$ induces
isomorphisms between residue fields. 

We conclude that $S^{sn}=S'$.
\end{proof}

\begin{lem}
$S^{sn}=\{m\in M;nm\in S\ \forall n\gg 0\}$. 
\end{lem}

\begin{proof} $\supseteq$: let $m\in M$ with $nm,(n+1)m\in S$ for some $n>0$.
Then $m\in \sigma_S$. Let $\sigma\prec\sigma_S$ such that $m\in \relint \sigma$. Then 
$m=(n+1)m-nm\in S_\sigma-S_\sigma$.

$\subseteq$: let $m\in S^{sn}$. Let $\sigma\prec\sigma_S$ such that $m\in \relint \sigma$. 
Let $(s_i)_i$ be a finite system of generators of $S\cap \sigma$. Then $m=\sum_i z_is_i$ for
some $z_i\in \Z$. Since $m\in \relint \sigma$, we can write $qm=\sum_i q_is_i$, with
$q,q_i\in \Z_{>0}$. There exists $l\ge 0$ such that $z_i+lq_i\ge 0$ for all $i$.
Then $lqm,(1+lq)m\in S$. Therefore $nm\in S$ for every $n\ge (lq-1)lq$.
\end{proof}

\begin{cor}
$X$ is seminormal if and only if $S\cap \relint \sigma=(S\cap \sigma-S\cap \sigma)\cap \relint \sigma$, for 
every face $\sigma \prec \sigma_S$.
\end{cor}

Recall~\cite{AB69} that the {\em weak normalization of $X$}, denoted $X^{wn}\to X$, is defined as a 
birational universal homeomorphism $f\colon Y \to X$, maximal with this property. It follows that 
$X^{wn}\to X$ is birational, and a topological homeomorphism.
We call $X$ {\em weakly normal} if its weak normalization is an isomorphism.

The normalization of $X$ factors as 
$
\bar{X}\to X^{wn}\stackrel{u}{\to} X^{sn}\to X,
$
where $u$ is a topological homeomorphism. If $\Char k=0$, then $u$ is an isomorphism.

\begin{prop} The weak normalization of $X$ is $\Spec k[S^{wn}] \to \Spec k[S]$, where 
$$
S^{wn}=\sqcup_{\sigma \prec \sigma_S} \cup_{e\ge 0} \{m\in  \bar{S}; p^em\in (S\cap \sigma -S\cap \sigma)\cap \relint \sigma \}.
$$
If $p=0$, we set $p^e=1$.
\end{prop}

\begin{proof}
The weak normalization is the spectrum of the ring
$$
R=\cap_{x\in X} \{f\in k[\bar{S}]; f_x^{p^e}\in \cO_{X,x}+\text{Rad} (\pi_*\cO_{\bar{X}})_x, \exists e\ge 0 \}.
$$
The proof is similar to that for seminormalization. We only need to use that if 
$\Lambda\subseteq \Lambda'\subseteq p^{-e}\Lambda$ are lattices, then 
$\Spec k[\Lambda'] \to \Spec k[\Lambda]$ is a universal homeomorphism (use relative Frobenius).
\end{proof}

Note that $S^{wn}=\cup_{e\ge 0}\{m\in \bar{S}; p^em\in S^{sn}\}$.

\begin{rem}
Let $S\subseteq S'$ be an inclusion of finitely generated semigroups,
such that $X'=\Spec k[S']\to \Spec k[S]=X$ is a finite morphism (i.e. for every $s'\in S'$,
there exists $n\ge 1$ such that $ns'\in S$). Then the seminormalization
of $X$ in $X'$ is associated to 
$$
\sqcup_{\sigma \prec \sigma_S}  S' \cap (S\cap \sigma-S\cap \sigma)\cap \relint \sigma 
$$
and the weak normalization of $X$ in $X'$ is associated to 
$$
 \sqcup_{\sigma \prec \sigma_S} \cup_{e\ge 0} \{m\in S' ; p^em\in (S\cap \sigma -S\cap \sigma) \cap \relint \sigma \}.
$$
\end{rem}

\begin{exmp}
Let $d$ be a positive integer. The extension $k[T]\subset k[T^d]$ is seminormal. 
It is weakly normal if and only if $p\nmid d$. Its weak normalization is $k[T]\subset k[T^{\frac{d}{d_p}}]\subset k[T^d]$, 
where $d_p$ is the largest divisor of $d$ which is not divisible by $p$.
\end{exmp}

\begin{exmp}
The semigroup $S=\{(x_1,x_2)\in \N^2;x_2>0\}\sqcup 2\N\times 0$ induces the $k$-algebra 
$k[S]\simeq k[X,Y,Z]/(ZX^2-Y^2)$. If $\Char k=2$, then $\Spec k[S]$ is seminormal but not weakly normal.
\end{exmp}

\begin{exmp}
Let $\dim S=1$. Then $\Spec k[S]$ is seminormal if and only if it is smooth, if and only if $S$ is isomorphic to 
$\N$ or $\Z$.
\end{exmp}

\begin{lem}\label{lag}
Let $\sigma\subset M_\R$ be a convex cone with non-empty interior. Then $M\cap \Int \sigma-M\cap \Int \sigma=M$.
\end{lem}

\begin{proof}
Let $m\in M$. Choose $m'\in M\cap \Int \sigma$. Then $m'+\epsilon m\in \Int \sigma$ for $0\le \epsilon\ll 1$.
Therefore $nm'+m\in \Int \sigma$ for $n\gg 0$. Then $m=(nm'+m)-(nm')\in M\cap \Int \sigma-M\cap \Int \sigma$.
\end{proof}

\begin{prop}[Classification of seminormal and weakly normal semigroups]
Let $\sigma\subset M_\R$ be a rational polyhedral cone, which generates $M_\R$.
There is a one to one correspondence between semigroups $S$ such that $S-S=M$, 
$\sigma_S=\sigma$ and 
$\Spec k[S]$ is seminormal, and collections $(\Lambda_\tau)_{\tau\prec \sigma}$ of 
sublattices of finite index $\Lambda_\tau \subseteq M\cap \tau-M\cap \tau$ 
such that $\Lambda_\sigma=M$ and $\Lambda_{\tau'}\subset \Lambda_\tau$ if $\tau' \prec \tau$.
The correspondence, and its inverse, is
$$
S\mapsto (S_\tau-S_\tau)_{\tau\prec \sigma}
\text{ and }
(\Lambda_\tau)_{\tau\prec \sigma} \mapsto \sqcup_{\tau \prec \sigma} \Lambda_\tau \cap \relint \tau.
$$
Moreover, $\Spec k[S]$ is weakly normal if and only if $p$ does not divide the index of the sublattice
$\Lambda_\tau=S_\tau-S_\tau \subseteq \bar{S}\cap \tau-\bar{S}\cap \tau$,
for every face $\tau\prec \sigma$.
\end{prop}

\begin{proof}
Use Lemma~\ref{lag} to show that the two are inverse. 
Moreover, if $\Spec k[S]$ is seminormal, it is weakly normal if and only if for all $\tau\prec \sigma$,
if $m\in (\bar{S}_\tau-\bar{S}_\tau)\cap \relint \tau$ and $pm\in \Lambda_\tau \cap \relint \tau$,
then $m\in \Lambda_\tau \cap \relint \tau$. By Lemma~\ref{lag}, this is equivalent to the index of 
the sublattice $S_\tau-S_\tau \subseteq \bar{S}\cap \tau-\bar{S}\cap \tau$ not being
divisible by $p$.
\end{proof}

Thus, a seminormal variety $X=\Spec k[S]$ is obtained from its normalization $\bar{X}=\Spec k[\bar{S}]$ by self-glueing
some invariant cycles $\bar{X}_\sigma \ (\sigma \prec\sigma_S$), according to the finite index sublattices
$\Lambda_\sigma \subseteq M\cap \sigma-M\cap \sigma$.


\subsection{Spectrum of a toric face ring~\cite{IR07}}


A {\em monoidal complex} $\cM=(M,\Delta,(S_\sigma)_{\sigma\in \Delta})$ consists of a
lattice $M$, a rational fan $\Delta$ with respect to $M$ (i.e. a finite collection 
of rational polyhedral cones in $M_\R$, such that every face of a cone of $\Delta$ is also 
in $\Delta$, and any two cones of $\Delta$ intersect along a common face), and a 
collection of finitely generated semigroups $S_\sigma\subseteq M\cap \sigma$, 
such that $S_\sigma$ generates $\sigma$ and $S_\tau=S_\sigma\cap \tau$ if $\tau\prec\sigma$.

If $0\in \Delta$, that is each cone of $\Delta$ admits the origin as a face, this is the definition 
introduced by Ichim and R\"omer~\cite{IR07}.

The {\em support} of $\cM$ is the set $|\cM|=\cup_{\sigma\in \Delta} S_\sigma \subseteq M$. The {\em toric face ring}
of $\cM$ is the $k$-algebra
$$
k[\cM]=\oplus_{m\in |\cM|}k\cdot \chi^m,
$$
with the following multiplication: $\chi^m\cdot \chi^{m'}$ is $\chi^{m+m'}$ if $m,m'$ are contained in some $S_\sigma$, and $0$ otherwise. Note that $k[\cM]\simeq \varprojlim_{\sigma\in \Delta}k[S_\sigma]$.

For the rest of this section, let $X=\Spec k[\cM]$. We call $X/k$ the {\em toric variety} associated to
the monoidal complex $\cM$. The torus $T=\Spec k[M]$ acts on $X$,
with $g^*(\chi^m)=g(m)\chi^m$. So the support of $\cM$ is recovered as the set of weights of the torus action.

\begin{exmp}
Let $M$ be a lattice, and $S\subseteq M$ a finitely generated semigroup such that $S-S=M$.
Let $\Delta$ be a subfan of the fan of faces of $\sigma_S$. Then 
$(M,\Delta,(S\cap \sigma)_{\sigma\in \Delta})$ is a monoidal complex, and $\Spec k[\cM]$ 
is a closed subvariety of $\Spec k[S]$ which is torus invariant.
\end{exmp}

\begin{exmp}\cite{Sta87} Let $\Delta$ be a rational fan with respect to $M$.
For $\sigma\in \Delta$, define $S_\sigma=M\cap \sigma$. This defines a monoidal
complex with toric face ring
$
k[M,\Delta]=\oplus_{m\in M\cap \Supp \Delta  }k\cdot \chi^m.
$
\end{exmp}

A $T$-invariant ideal $I\subseteq k[\cM]$ is radical if and only if $I=k[\cM\setminus A]$,
where $A=\cup_{\sigma\in \Delta'}S_\sigma$, where $\Delta'$ is a subfan of $\Delta$.
The quotient $k[\cM]/I$ is $k[\cM']$, where $\cM'=(M,\Delta',(S_\sigma)_{\sigma\in \Delta'})$.
In particular, $I\subseteq k[\cM]$ is a $T$-invariant prime ideal if and only if $I=k[\cM\setminus A]$
with $A=S_\sigma$ for some $\sigma\in \Delta$. The quotient $k[\cM]/I$ is $k[S_\sigma]$.

We obtain one to one correspondences between: i) $T$-invariant closed 
reduced subvarieties of $X$ and subfans of $\Delta$; ii) $T$-orbits of $X$ and the 
cones of $\Delta$. If $\sigma\in \Delta$, the ideal $k[\cM\setminus \sigma]$ defines a 
$T$-invariant closed reduced subvariety $X_\sigma\subseteq X$.
We have $\tau\prec \sigma$ if and only if $X_\tau\subseteq X_\sigma$.
The orbit corresponding to the cone $\sigma\in \Delta$ is 
$O_\sigma=X_\sigma \setminus \cup_{\tau\prec \sigma, \tau\ne \sigma }X_\tau$, and is isomorphic to 
the torus $\Spec k[S\cap \sigma-S\cap \sigma]$. We obtain
$$
X=\cup_{\sigma\in \Delta}X_\sigma=\sqcup_{\sigma\in \Delta}O_\sigma.
$$

The smallest cone of $\Delta$ is $\tau=\cap_{\sigma\in \Delta}\sigma$. The orbit $O_\tau$ is
the unique orbit which is closed. In particular, $0\in \Delta$ if and only if $\tau=0$, if and only if 
the torus action on $X$ has a (unique) fixed point, if and only if the cones of $\Delta$ are pointed
as in~\cite{IR07}.

Each $X_\sigma$ is an affine equivariant embedding of the torus $T_\sigma$, where 
$T_\sigma=\Spec k[S_\sigma-S_\sigma]$ is a quotient of $T$. The action of $T$ on $X_\sigma$
factors through the action of $T_\sigma$.

The irreducible components of $X$ are $X_F$, where $F$ are the {\em facets} of $\Delta$
(cones of $\Delta$ which are maximal with respect to inclusion). The torus $T$ acts on each 
irreducible component of $X$ 

The toric variety $X$ is irreducible if and only if $\Delta$ has a unique maximal cone, if and only if 
$X=\Spec k[S]$ is an equivariant torus embedding (see~\cite{IR07} for proofs of the above statements).

\begin{rem} A geometric characterization of $X=\Spec k[\cM]$ is as follows: $X/k$ is a reduced
affine algebraic variety, endowed with an action by a torus $T/k$, subject to the following axioms:
\begin{itemize}
\item[a)] $T$ acts on each irreducible component $X_i$ of $X$, and the action factors through 
a torus quotient $T\to T_i$ such that $T_i\subseteq X_i$ becomes an equivariant affine torus embedding.
\item[b)] The scheme intersection $X_i\cap X_j$ is reduced, and the induced action of $T$ on $X_i\cap X_j$
factors through a torus quotient $T\to T_{ij}$ such that $T_i\subseteq X_i$ becomes an equivariant 
affine torus embedding.
\end{itemize}
We have seen above that $X=\Spec k[\cM]$ satisfies properties a) and b). Conversely,
we recover the monomial complex as follows: $T=\Spec k[M]$ for some lattice $M$. Each 
irreducible component of $X$ is of the form $X_i=\Spec k[S_i]$ for some finitely generated semigroup 
$S_i\subseteq M$. Let $F_i\subseteq M_\R$ be the cone generated by $S_i$. Define $\Delta$ to 
be the collection of $F_i$ and their faces. Each $\sigma\in \Delta$ is a face of some $F_i$, and we 
set $S_\sigma=S_i\cap \sigma$. To verify that $\Delta$ is actually a fan, it suffices to show that
two maximal cones $F_i,F_j$ intersect along a common face. By b), $X_i\cap X_j=\Spec k[S_{ij}]$
for some finitely generated semigroup $S_{ij}\subseteq M$. Since $X_i\cap X_j$ is a $T$-invariant 
closed reduced subvariety of $X_i$, there exists a face $\tau_{ij}$ of $F_i$ such that $S_{ij}=S_i\cap \tau_{ij}$.
By a similar argument, there exists a face $\tau_{ji}$ of $F_j$ such that $S_{ij}=S_j\cap \tau_{ji}$.
Then $\tau_{ij},\tau_{ji}$ coincide, equal to the cone generated by $S_{ij}$, also equal to $F_i\cap F_j$. 
Therefore $F_i\cap F_j$ is a face in both $F_i$ and $F_j$.
\end{rem}

\begin{prop}[Nguyen~\cite{Ngu12}]\label{sen}
For $\sigma\in \Delta$, let $S^{sn}_\sigma= \sqcup_{\tau \prec \sigma}(S_\tau -S_\tau)\cap \relint \tau$ 
be the seminormalization of $S_\sigma$. 
Then $\cM^{sn}=(M,\Delta,(S^{sn}_\sigma)_{\sigma\in \Delta})$ is a monoidal complex, and 
the seminormalization of $X$ is 
$
\Spec k[\cM^{sn}]\to \Spec k[\cM].
$
\end{prop}

\begin{proof} For cones $\tau,\sigma\in \Delta$, we have $\sigma\cap \relint \tau \ne \emptyset$ if and only if 
$\tau\prec \sigma$. Therefore $S^{sn}_\tau=\tau\cap S^{sn}_\sigma$ if $\tau\prec \sigma$. We conclude that 
$\cM^{sn}=(M,\Delta,(S^{sn}_\sigma)_{\sigma\in \Delta})$ is a monoidal complex.

Let $\{F\}$ be the facets of $\Delta$. The normalization of $X$ is 
$$
\bar{X}=\sqcup_F \Spec k[(S_F-S_F)\cap F].
$$
The torus $T$ acts on $\bar{X}$ too, and is compatible with $\pi\colon \bar{X}\to X$. The seminormalization is the spectrum of the ring
$$
R=\cap_{x\in X}\{ f\in \cO(\bar{X}); f_x\in \cO_{X,x}+\text{Rad}(\cO_{\bar{X}})_x \}.
$$
The torus $T$ acts on $R$, and therefore $R=\prod_F k[S'_F]$ for certain semigroups $S'_F\subseteq (S_F-S_F)\cap F$, which remains 
to be identified.

Let $\sigma\in \Delta$. It defines a $T$-invariant cycle $X_\sigma \subset X$. Its preimage $\pi^{-1}(X_\sigma)$ is 
$\sqcup_F (\bar{X}_F)_{\sigma\cap F}$. So if $x$ is the generic point of $X_\sigma$, $\pi^{-1}(x)$ consists of the 
generic points of $X_\sigma \subset \Spec k[F\cap (S_F-S_F)]$, after all facets $F$ which contain $\sigma$.

We deduce that $f=(\chi^{m_F})_F$ satisfies the glueing condition over the generic point of $X_\sigma$ if and only if 
either $M_F\notin \sigma$ for all $F\supseteq \sigma$, or there exists $m\in (S_\sigma-S_\sigma)\cap \sigma$ such that
$m_F=m$ for all $F\supseteq \sigma$. Choose a component $F_1$, and let $m_{F_1}\in \relint \tau$. For $\sigma=\tau$,
we obtain $m_F=m\in (S_\tau-S_\tau)\cap \tau$ for all $F\supseteq \tau$. But $m\in F$ if and only if $\tau\subseteq F$.
Therefore $f=\pi^*\chi^m$, with 
$$
m\in \sqcup_{\sigma\in \Delta} (S_\sigma-S_\sigma)\cap \relint \sigma.
$$
One checks that it's enough to glue only over invariant points. Therefore $X^{sn}=\Spec k[\cM^{sn}]$.
\end{proof}

Similarly, we obtain 

\begin{prop}\label{wen} 
For each $\sigma\in \Delta$, let $S^{wn}_\sigma$ be the weak-normalization of $S_\sigma$:
$$
S^{wn}_\sigma=\sqcup_{\tau\prec \sigma}\cup_{e\ge 0} \{m\in \overline{S_\sigma} ;p^em\in  (S_\tau-S_\tau) \cap \relint \tau \}.
$$
Then $\cM^{wn}=(M,\Delta,(S^{wn}_\sigma)_{\sigma\in \Delta})$ is a monoidal complex, and 
the weak-normalization of $X$ is 
$$
\Spec k[\cM^{wn}]\to \Spec k[\cM].
$$
\end{prop}

\begin{cor}\label{fwn}
$X$ is seminormal (resp. weakly normal) if and only if $X_F$ is seminormal (resp. weakly normal)
for every facet $F$ of $\Delta$, if and only if $X_\sigma$ is seminormal (resp. weakly normal) for 
every face $\sigma\in \Delta$.
\end{cor}

In particular, if $X$ is seminormal (weakly normal), so is any union of torus invariant closed subvarieties of $X$.

\begin{prop}[Classification of seminormal and weakly normal toric face rings]
Let $M$ be a lattice and $\Delta$ a finite rational fan in $M_\R$.
There is a one to one correspondence between collections of semigroups
$(S_\sigma)_{\sigma}\in \Delta$ such that $(M,\Delta,(S_\sigma)_{\sigma\in \Delta})$ is 
a monoidal complex with $\Spec k[\cM]$ seminormal, and collections $(\Lambda_\sigma)_{\sigma\in \Delta}$ 
of sublattices of finite index $\Lambda_\sigma \subseteq M\cap \sigma-M\cap \sigma$ 
such that $\Lambda_\tau \subset \Lambda_\sigma$ if $\tau \prec \sigma$.
The correspondence, and its inverse, is
$$
(S_\sigma)_\sigma  \mapsto (S_\sigma-S_\sigma)_\sigma 
\text{ and }
(\Lambda_\sigma)_\sigma \mapsto (\sqcup_{\tau \prec \sigma} \Lambda_\tau \cap \relint \tau)_\sigma.
$$
Moreover, $\Spec k[\cM]$ is weakly normal if and only if $p$ does not divide the index of the 
sublattice $S_\sigma-S_\sigma\subset (S_F-S_F)\cap \sigma-(S_F-S_F)\cap \sigma$, for 
every $\sigma\prec F$ in $\Delta$, with $F$ a facet of $\Delta$. 
\end{prop}

\begin{rem}\label{al}
Let $x$ be a point which belongs to the closed orbit of $X$. Then $X$ is seminormal
(resp. weakly normal) if and only if $\cO_{X,x}$ is seminormal (resp. weakly normal). 
Indeed, the direct implication  is clear. For the converse, note that the proofs of Propositions~\ref{sen} 
and ~\ref{wen} show that $X$ is seminormal (resp. weakly normal) if and only if so are $\cO_{X,X_\sigma}$ for 
all $\sigma\in \Delta$. Since $x\in X_\sigma$ for all $\sigma\in \Delta$, the converse holds as well.
\end{rem}

\begin{rem}\label{deco} 
Consider the germ of $X$ near a closed point $x$. There exists a unique cone $\tau\in \Delta$
such that $x\in O_\tau$. The smallest $T$-invariant open subset of $X$ which contains $x$ is 
$U=\sqcup_{\tau\prec \sigma\in \Delta}O_\sigma$. If we choose $s\in S_\tau\cap \relint \tau$, then 
$U$ coincides with the principal open set $D(\chi^s)$. We deduce that $U=\Spec k[\cM_x]$, where 
$\cM_x$ is the monoidal complex $(M,\{\sigma-\tau\}_{\tau\prec \sigma\in \Delta},(S_\sigma-S_\tau)_{\tau\prec \sigma\in \Delta})$.
We have an isomorphism of germs $(X,x)=(U,x)$, and $x$ is contained in the orbit associated to 
$\tau-\tau$, the smallest cone of $\Delta(\cM_x)$.

Consider the quotient 
$
\pi\colon M\to M'=M/(M\cap \tau-M\cap \tau).
$ 
For $\tau\prec \sigma\in \Delta$, denote $\sigma'=\pi(\sigma)\subseteq M'_\R$ and $S_{\sigma'}=\pi(S_\sigma)$.
Then $\pi^{-1}(\sigma')=\sigma-\tau$, the cone generated by $S_\sigma-S_\tau$. On the other hand, 
$\pi^{-1}(S_{\sigma'})=S_\sigma+M\cap \tau-M\cap \tau$ is usually larger than $S_\sigma-S_\tau$.

Suppose $S_\sigma=M\cap \sigma$ for every $\tau\prec \sigma\in \Delta$. Then $S_\sigma-S_\tau=\pi^{-1}(S_{\sigma'})$.
The choice of a splitting of $\pi$ induces an isomorphism 
$
X\simeq T'' \times \Spec k[\cM'],
$
where $T''$ is the torus $\Spec k[M\cap \tau-M\cap \tau]$ and $\cM'$ is the monoidal complex
$(M', \{\sigma'\},\{S_{\sigma'}\})$. We have $0\in \Delta(\cM')$, so 
$\Spec k[\cM']$ has a fixed point $x'$. The isomorphism maps $x$ onto $(x'',x')$, where $x''\in T''$ is a closed point.
In particular, $(X,x)\simeq (T'',x'')\times (\Spec k[\cM'],x')$.
\end{rem}


\section{Du Bois complex for the spectrum of a toric face ring}


Let $X=\Spec k[\cM]$ be the affine variety associated
to a monoidal complex $\cM=(M,\Delta,(S_\sigma)_{\sigma\in \Delta})$. Suppose $X$ is weakly normal.

For $m\in \cup_{\sigma\in \Delta} S_\sigma$, denote by $\sigma_m$ the unique cone of $\Delta$ 
which contains $m$ in its relative interior. Denote by $V_m$ the invariant regular $1$-forms
on the torus $\Spec k[S_{\sigma_m}-S_{\sigma_m}]$. 
For each $p$, denote 
$$
A^p(X)=\oplus_{m\in \cup_{\sigma\in \Delta} S_\sigma} \chi^m\cdot \wedge^pV_m.
$$
If $m,m'\in S_\sigma$ for some $\sigma\in \Delta$, then $\sigma_m$ is a face of $\sigma_{m+m'}$, hence
$V_m\subseteq V_{m+m'}$. Therefore $A^p(X)$ becomes a $\Gamma(X,\cO_X)$-module in a natural way:
$\chi^{m'}\cdot (\chi^m\omega_m)=(\chi^{m'}\cdot \chi^m)\omega_m$. It induces a coherent $\cO_X$-module, 
denoted $\tilde{\Omega}_X^p$, with 
$
\Gamma(X, \tilde{\Omega}_X^p )=A^p(X).
$

\begin{lem}\label{fct}
For every morphism $f\colon X'\to X$ from a smooth variety $X'$, we can naturally define a pullback 
homomorphism $f^*\colon A^p(X)\to \Gamma(X',\Omega^p_{X'})$. Moreover, each commutative diagram
\[ 
\xymatrix{
         X' \ar[d]_f & Y' \ar[l]_v \ar[d]^{f'}  \\
         X                  &  X_\tau  \ar@{_{(}->}[l]               
} \]
with $X',Y'$ smooth and $\tau\in \Delta$, induces a commutative diagram
\[ 
\xymatrix{
         \Gamma(X',\Omega^p_{X'})   \ar[r]^{v^*}  & \Gamma(Y',\Omega^p_{Y'})   \\
         A^p(X)      \ar[r]^{|_{X_\tau}}               \ar[u]^{f^*}         &  A^p(X_\tau)   \ar[u]_{{f'}^*}                
} \]
where $\chi^m\omega_m|_{X_\tau}$ is $\chi^m\omega_m$ if $m\in \tau$, and $0$ otherwise.
\end{lem}

\begin{proof} Let $f\colon X'\to X$ be a morphism from a smooth variety $X'$. To define $f^*$,
we may suppose $X'$ is irreducible. Let $\tau$ be the smallest cone of $\Delta$ such that
$f(X')\subseteq X_\tau$. In particular, $f(X')$ intersects the orbit $O_\tau$. Let $f'\colon X'\to X_\tau$
be the induced morphism. 

Let $\omega\in A^p(X)$. Let $\omega_\tau=\omega |_{X_\tau}\in A^p(X_\tau)$ be its combinatorial restriction,
defined above. It is a regular differential $p$-form on the orbit $O_\tau$,
which is smooth, being isomorphic to a torus. Therefore ${f'}^*(\omega_\tau)$ is a well defined rational differential
$p$-form on $X'$ (regular on ${f'}^{-1}(O_\tau)$). Define 
$$
f^*\omega={f'}^*(\omega_\tau). 
$$
We claim that ${f'}^*(\omega_\tau)$ is regular everywhere on $X'$.
Indeed, choose a toric desingularization $\mu_\tau\colon Y_\tau\to X_\tau$. By Hironaka's resolution of the
indeterminacy locus of the rational map $X'\dashrightarrow Y_\tau$, we obtain a commutative diagram
\[ 
\xymatrix{
         Y' \ar[d]_\mu \ar[r]^h & Y_\tau  \ar[d]^{\mu_\tau}  \\
         X'    \ar[r]^{f'}               &  X_\tau               
} \]
By the combinatorial description of differential forms on the smooth toric variety $Y_\tau$, 
the rational form $\mu_\tau^*(\omega_\tau)$ is in fact regular everywhere on $Y_\tau$.
Then $h^*\mu_\tau^*(\omega_\tau)$ is regular on $Y'$. Therefore $\mu^*({f'}^*\omega_\tau)$
is regular on $Y'$. But $\mu$ is a proper birational contraction and $X',Y'$ are smooth, so
$\Omega^p_{X'}=\mu_*(\Omega^p_{Y'})$. Therefore the rational form ${f'}^*\omega_\tau$ is 
in fact regular on $X'$.

It remains to verify the commutativity of the square diagram. We may suppose $X'$ and $Y'$
are irreducible. Let $f(X')\subseteq X_\sigma$ and $f(Y')\subseteq X_\tau$, with $\sigma$ and $\tau$
minimal with this property. We obtain a commutative diagram 
\[ 
\xymatrix{
         X' \ar[d]_f & Y' \ar[l]_v \ar[d]^{f'}  \\
         X_\sigma                  &  X_\tau  \ar@{_{(}->}[l]               
} \]
We may replace $f$ by a toric desingularization of $X_\sigma$.
Then $v(Y')$ is contained in $f^{-1}(X_\tau)$, which is a union of closed invariant cycles
of $X'$. Each of these closed invariant cycles is smooth, since $X'$ is smooth. We may
replace $Y'$ by an invariant closed cycle of $X'$ which contains $v(Y')$. It remains to check
the claim for the special type of diagrams
\[ 
\xymatrix{
         X' \ar[d]_f & Y' \ar@{_{(}->}[l] \ar[d]^{f'}  \\
         X_\sigma                  &  X_\tau  \ar@{_{(}->}[l]               
} \]
where $f$ is a toric desingularization, $Y'\subset X'$ is a closed invariant cycle, and 
$f'(Y')\cap O_\tau\ne \emptyset$. By the explicit combinatorial formula for differential forms on smooth
toric varieties, the diagram
\[ 
\xymatrix{
         \Gamma(X',\Omega^p_{X'})   \ar[r]^{|_{Y'}}  & \Gamma(Y',\Omega^p_{Y'})   \\
         A^p(X_\sigma)  \ar[r]^{|_{X_\tau}} \ar[u]^{f^*}         &  A^p(X_\tau)   \ar[u]_{{f'}^*}                
} \]
is commutative.
\end{proof}

\begin{lem}\label{fct'} Consider a commutative diagram
\[ 
\xymatrix{
         X' \ar[d]_{f'} & X'' \ar[l]_v \ar[dl]^{f''}  \\
         X                 &            
} \]
with $X',X''$ smooth. Then the induced diagram of pullbacks 
\[ 
\xymatrix{
         \Gamma(X',\Omega^p_{X'}) \ar[r]^{v^*}  & \Gamma(X'',\Omega^p_{X''})  \\
         A^p(X)   \ar[u]^{{f'}^*}       \ar[ur]_{{f''}^*}         &            
} \]
is commutative.
\end{lem}

\begin{proof}
We may suppose $X''$ is irreducible. Then there exists $\tau\in \Delta$ such that 
$f''(X'')\subset X_\tau$. We obtain a diagram 
\[ 
\xymatrix{
         X' \ar[d] & X'' \ar[l] \ar[dl] \ar[d] \\
         X                 &   X_\tau  \ar@{_{(}->}[l]        
} \]
and the claim follows by applying Lemma~\ref{fct} twice.
\end{proof}

\begin{thm}\label{mv}
Let $\epsilon\colon X_\bullet\to X$ be a smooth simplicial resolution. Then 
the natural homomorphism 
$\tilde{\Omega}^p_X\to R\epsilon_*(\Omega^p_{X_\bullet})$ is a quasi-isomorphism
(i.e. $\tilde{\Omega}^p_X\isoto \epsilon_*(\Omega^p_{X_\bullet})$ and 
$R^i\epsilon_*(\Omega^p_{X_\bullet})=0$ for $i> 0$).
\end{thm}

\begin{proof} Let $\delta_0,\delta_1\colon X_1\to X_0$ be the two face morphisms.
Then $\delta_0^*\epsilon_0^*=\epsilon_1^*=\delta_1^*\epsilon_0^*$, by Lemma~\ref{fct'}.
Therefore the pullback homomorphism $\epsilon_0^*$ maps $\tilde{\Omega}^p_X$ into 
$\Ker({\epsilon_0}_* \Omega^p_{X_0}  \rightrightarrows {\epsilon_1}_* \Omega^p_{X_1})=\epsilon_*(\Omega^p_{X_\bullet})$.
This defines a natural homomorphism $\tilde{\Omega}^p_X\to R\epsilon_*(\Omega^p_{X_\bullet})$.
We show that this is a quasi-isomorphism, by induction on $\dim X$.

Let $\Sigma$ be the (toric) boundary of $X$. The restriction $\tilde{\Omega}^p_X\to \tilde{\Omega}^p_\Sigma$ is surjective.
Denote its kernel by $\tilde{\Omega}^p_{(X,\Sigma)}$. 
Denote by $\underline{\Omega}^p_X$ the complex on the right hand side. We obtain a commutative diagram
 \[ 
\xymatrix{
    0 \ar[r] & \underline{\Omega}^p_{(X,\Sigma)} \ar[r] & \underline{\Omega}^p_X  \ar[r] & \underline{\Omega}^p_\Sigma \ar[r] & 0 \\
    0 \ar[r] & \tilde{\Omega}^p_{(X,\Sigma)} \ar[r] \ar[u]^\alpha & \tilde{\Omega}^p_X  \ar[r] \ar[u]^\beta & \tilde{\Omega}^p_\Sigma \ar[r] \ar[u]^\gamma & 0      
} \]
where the bottom raw is exact, and the top raw is an exact triangle in the derived category.
Since $\Sigma$ is again weakly normal, $\gamma$ is a quasi-isomorphism by induction on dimension.
If $\alpha$ is a quasi-isomorphism, then $\beta$ is a quasi-isomorphism.

We claim that $\alpha$ is a quasi-isomorphism. Indeed, let $\pi\colon \bar{X}\to X$ be the 
normalization. Each component of $\bar{X}$ is an affine, normal toric variety. We construct 
a desingularization $\bar{f}\colon Y\to \bar{X}$ by choosing a toric desingularization for each component of 
$\bar{X}$. Let $f=\pi \circ \bar{f}\colon Y\to X$ be the induced
desingularization. Let $\Sigma'$ and $\bar{\Sigma}$ be the (toric) boundaries of $Y$ and $\bar{X}$
respectively. Since $f\colon Y\setminus \Sigma'\to X\setminus \Sigma$ is an isomorphism, we 
obtain a quasi-isomorphism 
$
\underline{\Omega}^p_{(X,\Sigma)}\to Rf_* \underline{\Omega}^p_{(Y,\Sigma')}
$
(see the proof of~\cite[Proposition 3.9]{DB81}).
Since $Y$ is smooth and $\Sigma'$ is a normal crossings divisor in $Y$,
$\Omega^p_{(Y,\Sigma')}\to   \underline{\Omega}^p_{(Y,\Sigma')}$ is a quasi-isomorphism. 
By~\cite[Proposition 1.8]{Dan91}, $\tilde{\Omega}^p_{(\bar{X},\bar{\Sigma})}\to R\bar{f}_*\Omega^p_{(Y,\Sigma')}$
is a quasi-isomorphism. Since $\pi$ is finite,  
$\pi_* \tilde{\Omega}^p_{(\bar{X},\bar{\Sigma})}\to Rf_*\Omega^p_{(Y,\Sigma')}$
is a quasi-isomorphism. As $X$ is weakly normal, we see combinatorially that 
$\tilde{\Omega}^p_{(X,\Sigma)}=\pi_* \tilde{\Omega}^p_{(\bar{X},\bar{\Sigma})}$
(since if $\Spec k[S]$ is weakly normal, then $S\cap \relint \sigma_S=(S-S)\cap \relint \sigma_S$).
From the commutative diagram
\[ 
\xymatrix{
         \pi_*\tilde{\Omega}^p_{(\bar{X},\bar{\Sigma})}   \ar[r]  & Rf_*\underline{\Omega}^p_{(Y,\Sigma')}     \\
        \tilde{\Omega}^p_{(X,\Sigma)}   \ar[r]  \ar[u]        &  \underline{\Omega}^p_{(X,\Sigma)}   \ar[u]               
} \]
we conclude that $\tilde{\Omega}^p_{(X,\Sigma)}\to \underline{\Omega}^p_{(X,\Sigma)}$ is a quasi-isomorphism.
\end{proof}

Let $d\colon A^p(X)\to A^{p+1}(X)$ be the $k$-linear map such that $d(\chi^m\omega_m)=
\chi^m\cdot (\frac{d\chi^m}{\chi^m}\wedge \omega_m)$.
It induces a structure of complex with $k$-linear differential $\tilde{\Omega}^*_X$. Let $F$ be its naive filtration.

\begin{cor}\label{cDB}
Let $\epsilon\colon X_\bullet\to X$ be a smooth simplicial resolution. Then the natural homomorphism 
$(\tilde{\Omega}^*_X,F)\to R\epsilon_*(\Omega^*_{X_\bullet},F)$ is a filtered quasi-isomorphism.
\end{cor}

Note that $\cO_X=\tilde{\Omega}^0_X$. 

\begin{lem}
Let $f\colon X'\to X$ be a desingularization, let $X''\to X'\times_X X'$ be a desingularization. We obtain a
commutative diagram
$$
\xymatrix{
X'  \ar[d]_{f'} & X'' \ar[dl]^{f''} \ar@<1ex>[l]^{p_2}\ar@<-1ex>[l]_{p_1}   \\
X  &
}
$$
Then $\tilde{\Omega}^p_X\isoto \Ker(f'_*\Omega^p_{X'} \rightrightarrows f''_*\Omega^p_{X''})
=\{\omega'\in f'_*\Omega^p_{X'}; p_1^*\omega'=p_2^*\omega' \}$.
\end{lem}

\begin{proof} Let $X_0=X'$, $\epsilon_0=f'$. Let $X_1=X'' \sqcup X'$, let 
$\delta_0,\delta_1\colon X_1\to X_0$ be the identity on $X'$, and $p_1,p_2$
respectively on $X''$. Let $s_0\colon X_0\to X_1$ be the extension of the identity of $X'$.
Let $\epsilon_1=f''$. Both desingularizations are proper and surjective, so we obtain a 
$1$-truncated smooth proper hypercovering
$$
\xymatrix{
X_0 \ar[d]_{\epsilon_0} \ar[r] & X_1 \ar[dl]^{\epsilon_1} \ar@<1ex>[l]^{\delta_1}\ar@<-1ex>[l]_{\delta_0}   \\
X  &
}
$$
We can extend the $1$-truncated augmented simplicial object to a smooth proper hypercovering 
$\epsilon \colon X_\bullet\to X$ (see~\cite[6.7.4]{Del74}).
By Theorem~\ref{mv}, $\tilde{\Omega}^p_X\isoto \epsilon_*(\Omega^p_{X_\bullet})$. But
$$
\epsilon_*(\Omega^p_{X_\bullet})=\Ker({\epsilon_0}_*\Omega^p_{X_0} \rightrightarrows {\epsilon_1}_*\Omega^p_{X_1})
=\Ker(f'_*\Omega^p_{X'} \rightrightarrows f''_*\Omega^p_{X''}).
$$
\end{proof}

In particular, $\tilde{\Omega}^p_X$ coincides with the sheaf of $h$-differential forms~\cite{HJ14}.


\subsection{Toric pairs}


Let $X=\Spec k[\cM]$ be a weakly normal affine variety associated
to a monoidal complex $\cM=(M,\Delta,(S_\sigma)_{\sigma\in \Delta})$. The torus $\Spec k[M]$ acts
on $X$. Let $Y\subset X$ be an invariant closed subscheme, with reduced structure.
Then $Y=\Spec k[\cM']$, where $\cM'=(M,\Delta',(S_\sigma)_{\sigma\in \Delta'})$ and $\Delta'$ is a subfan of $\Delta$, 
is also weakly normal. The restriction $A^p(X)\to A^p(Y)$ is surjective. Denote the kernel by $A^p(X,Y)$.
We have 
$$
A^p(X,Y)=\oplus_{m\in \cup_{\sigma\in \Delta}S_\sigma\setminus \cup_{\tau\in \Delta'}S_\tau}
\chi^m\cdot \wedge^pV_m.
$$
Denote by $\tilde{\Omega}^p_{(X,Y)}$ the coherent $\cO_X$-module induced by $A^p(X,Y)$. We obtain a 
short exact sequence
$$
0\to \tilde{\Omega}^p_{(X,Y)} \to \tilde{\Omega}^p_{X} \stackrel{|_{Y}}{\to} \tilde{\Omega}^p_Y\to 0.
$$
We obtain a differential complex $\tilde{\Omega}^*_{(X,Y)}$, and if we denote by $F$ the naive filtration,
we obtain by Corollary~\ref{cDB} a filtered quasi-isomorphism
$$
(\tilde{\Omega}^*_{(X,Y)},F) \to (\underline{\Omega}^*_{(X,Y)},F).
$$ 
Note that $\cI_{Y\subset X}=\tilde{\Omega}^0_{(X,Y)}$.

\begin{rem}
$\Gamma(X,\tilde{\Omega}^p_{(X,Y)})=\oplus_{\sigma\in \Delta,X_\sigma\not\subset Y}
\Gamma(X_\sigma,\tilde{\Omega}^p_{(X_\sigma,\partial X_\sigma)})$, where $\partial X_\sigma=X_\sigma\setminus O_\sigma$
is the toric boundary of the irreducible toric variety $X_\sigma$.
\end{rem}

\begin{rem}
The sheaf of $h$-differentials can be computed without the weakly normal assumption.
Let $X=\Spec k[\cM]$ be the variety associated to a monoidal complex. 
Let $f\colon X^{wn}\to X$ be the weak normalization, described in Proposition~\ref{wen}. Then 
$\tilde{\Omega}^p_X=f_* \tilde{\Omega}^p_{X^{wn}}$.
If $Y\subset X$ is a union of closed torus invariant subvarieties, then $f^{-1}(Y)$ is weakly normal,
hence $f^{-1}(Y)=Y^{wn}$. We obtain $\tilde{\Omega}^p_{(X,Y)}=f_* \tilde{\Omega}^p_{(X^{wn},Y^{wn})}$.
\end{rem}


\section{Weakly toroidal varieties}


Let $k$ be an algebraically closed field of characteristic zero. An algebraic variety $X/k$
has {\em weakly toroidal singularities} if for every closed point $x\in X$, 
there exists an isomorphism of complete local $k$-algebras $\cO_{X,x}^\wedge\simeq \cO_{X',x'}^\wedge$,
where $X'=\Spec k[\cM]$ is weakly normal, associated to a monoidal complex 
$\cM=(M,\Delta,(S_\sigma)_{\sigma\in \Delta})$, and $x'$ is a closed point contained in the 
closed orbit of $X'$. We say that $(X',x')$ is a local model for $(X,x)$.

\begin{exmp}
Let $\cM=(M,\Delta,(S_\sigma)_{\sigma\in \Delta})$ be a monoidal complex.
Then $X=\Spec k[\cM]$ has weakly toroidal singularities if and only if $X$ is weakly normal
(by Remarks~\ref{al} and~\ref{deco}).
\end{exmp}

\begin{rem} Suppose a local model of $(X,x)$ satisfies $S_\sigma=M\cap \sigma$ for all $\sigma\in \Delta$.
Then we can find another local model such that $x'$ is a fixed point of the torus action (use the discussion
on germs at the end of Section 2). In particular, for Danilov's toroidal singularities~\cite{Dan91} and Ishida's 
polyhedral singularities~\cite{Ish85} we can always find local models near a fixed point.
\end{rem}

Let $X/k$ have weakly toroidal singularities. Then $X$ is normal if and only if $X$
is toroidal in the sense of Danilov, that is the local models are $(X',x')$ with $X'$
an affine toric normal variety, and $x'$ a torus invariant closed point of $X'$.

\begin{thm}\label{mvtl}
Let $X$ have weakly toroidal singularities. Let $\epsilon\colon X_\bullet\to X$ be a smooth simplicial resolution. 
Then $\tilde{\Omega}^p_X\to R\epsilon_*(\Omega^p_{X_\bullet})$ is a quasi-isomorphism
(i.e. $\tilde{\Omega}^p_X\isoto \epsilon_*(\Omega^p_{X_\bullet})$ and 
$R^i\epsilon_*(\Omega^p_{X_\bullet})=0$ for $i>0$).
\end{thm}

\begin{proof}
The statement is local, and invariant under \'etale base change. By~\cite{Ar69}, we
may suppose $X=\Spec k[\cM]$ is a weakly normal local model. Then 
we may apply Theorem~\ref{mv}.
\end{proof}

Thus, the filtered complex $(\tilde{\Omega}^*_X,F)$, with $F$ the naive filtration, is a canonical choice 
for the Du Bois complex of $X$.

\begin{cor}
Let $X$ have weakly toroidal singularities. Then $X$ has Du Bois singularities.
\end{cor}

\begin{proof} We claim that $\cO_X=\tilde{\Omega}^0_X$.
The statement is local, and invariant under \'etale base change. By~\cite{Ar69}, we
may suppose $X=\Spec k[\cM]$ is a weakly normal local model.
By definition, $A^0(X)=\Gamma(X,\cO_X)$. Therefore the claim holds.
\end{proof}

\begin{lem}[Poincar\'e lemma]\label{pol}
Let $X/\C$ have weakly toroidal singularities. Then $\C_{X^{an}}\to \tilde{\Omega}^*_{X^{an}}$
is a quasi-isomorphism.
\end{lem}

\begin{proof} Let $\epsilon\colon X_\bullet\to X$ be a smooth simplicial resolution. 
Consider the commutative diagram
\[ 
\xymatrix{
         R\epsilon_*\C_{X_\bullet^{an}}   \ar[r]  & R\epsilon_*\Omega^*_{X_\bullet^{an}}     \\
        \C_{X^{an}}  \ar[r]  \ar[u]        &  \tilde{\Omega}^*_{X^{an}}   \ar[u]               
} \]
The left vertical arrow is a quasi-isomorphism from the definition of $\epsilon$. The right
vertical arrow is a quasi-isomorphism by Theorem~\ref{mvtl}. The top horizontal arrow is
a quasi-isomorphism, by the Poincar\'e lemma on each component $X_n^{an}\ (n\ge 0)$.
Therefore the bottom horizontal arrow is a quasi-isomorphism.
\end{proof}

\begin{thm}\label{E1deg}
Let $X/\C$ be proper, with weakly toroidal singularities. Then the spectral sequence
$$
E^{pq}_1=H^q(X,\tilde{\Omega}^p_X)\Longrightarrow \Gr^p_F H^{p+q}(X^{an};\C)
$$
degenerates at $E_1$, and converges to the Hodge filtration on the cohomology groups of $X^{an}$.
\end{thm}

\begin{proof} This follows from Theorem~\ref{mvtl} and~\cite{DB81}.
More precisely, let $\epsilon\colon X_\bullet\to X$ be a smooth simplicial resolution.
Then $(\Omega_{X_\bullet}^*,W,F)$, with $W$ the trivial filtration, is the analytical part of a 
cohomological moved Hodge $\Z$-complex on $X_\bullet$ (see~\cite[Example 8.1.12]{Del74} with 
$Y_\bullet=\emptyset$). Since $X_\bullet$ is proper, we obtain by~\cite[Theorem 8.1.15.(i), Scolie 8.1.9.(v)]{Del74}
a spectral sequence 
$$
E^{pq}_1=H^q(X_\bullet,\tilde{\Omega}^p_{X_\bullet})\Longrightarrow \Gr^p_F H^{p+q}(X_\bullet^{an};\C)
$$
which degenerates at $E_1$, and converges to the Hodge filtration on the cohomology groups of $X_\bullet^{an}$.
By Theorem~\ref{mvtl} and Lemma~\ref{pol}, this pushes down on $X$ to our claim.
\end{proof}

Finally, we check that $\tilde{\Omega}^p_X$ coincides with the sheaves defined by Danilov~\cite{Dan91} and Ishida~\cite{Ish85}:
\begin{itemize}
\item Suppose $X$ is toroidal. If $f\colon X'\to X$ is a desingularization and $w\colon U\subseteq X$ is the inclusion of the 
smooth locus, then  
$
\tilde{\Omega}^p_X=f_*(\Omega^p_{X'})=w_*(\Omega^p_U).
$ 

Indeed, by~\cite{Ar69} and \'etale base change, we may suppose $X$ is an affine toric normal variety. 
We may replace $f$ by a toric desingularization. Danilov shows in~\cite[Lemma 1.5]{Dan78} that 
$A^p(X)=\Gamma(X',\Omega^p_{X'})=\Gamma(U,\Omega^p_U)$. 

\item Suppose $X$ is a torus invariant closed reduced subvariety of an affine toric normal variety. 
Then $A^p(X)$ coincides with Ishida's module $\tilde{\Omega}^p_{B(\Phi)}$ defined in~\cite[page 119]{Ish85}. 
\end{itemize}


\subsection{Weakly toroidal pairs}


A {\em weakly toroidal pair} $(X,Y)$ consists of a weakly normal algebraic variety $X/k$ and a closed
reduced subvariety $Y\subseteq X$ such that for every closed point $x\in X$ there exists an isomorphism
of complete local $k$-algebras $\cO_{X,x}^\wedge\simeq \cO_{X',x'}^\wedge$, mapping
$\cI_{Y,x}^\wedge$ onto $\cI_{Y',x'}^\wedge$, 
where $X'=\Spec k[\cM]$ is the affine variety associated to
some monoidal complex $\cM=(M,\Delta,(S_\sigma)_{\sigma\in \Delta})$,
$Y'=\cup_{\sigma\in \Delta'}X'_\sigma\subseteq X'$ is a closed reduced subvariety which is invariant under
the action of the torus $\Spec k[M]$, and $x'$ is a closed point contained in the closed orbit of $X'$. 

\begin{exmp}
Let $\cM=(M,\Delta,(S_\sigma)_{\sigma\in \Delta})$ and $\Delta'$ a subfan of $\Delta$.
Consider $X=\Spec k[\cM]$ and $Y=\cup_{\sigma\in \Delta'}X_\sigma$. Then $(X,Y)$ is 
a weakly toroidal pair if and only if $X$ is weakly normal.
\end{exmp}

\begin{exmp}
Suppose $X$ is weakly toroidal. Let $\Sing X$ and $C$ be the singular and non-normal locus of $X$, respectively.
Then $(X,\Sing X)$ and $(X,C)$ are weakly toroidal pairs.
\end{exmp}

If $(X,Y)$ is a weakly toroidal pair, then $X$ and $Y$ are weakly toroidal.

Let $(X,Y)$ be a weakly toroidal pair. Define $\tilde{\Omega}^p_{(X,Y)}=\Ker(\tilde{\Omega}^p_X\to \tilde{\Omega}^p_Y)$.

\begin{lem} Let $(X,Y)$ be a weakly toroidal pair. We have a short exact sequence 
$$
0\to \tilde{\Omega}^p_{(X,Y)} \to \tilde{\Omega}^p_X\to \tilde{\Omega}^p_Y \to 0.
$$
\end{lem}

\begin{proof}
By \'etale base change and~\cite{Ar69}, we may suppose $(X,Y)$ is a local model. Then 
$A^p(X)\to A^p(Y)$ is surjective, by the combinatorial formulas for the two modules.
\end{proof}

\begin{thm}\label{mvtlpr}
Let $(X,Y)$ be a weakly toroidal pair. Let $\epsilon\colon X_\bullet\to X$ be a smooth simplicial resolution,
such that $\epsilon^{-1}(Y)=Y_\bullet$ is locally on $X_\bullet$ either empty, or a normal crossing divisor.
Then $\tilde{\Omega}^p_{(X,Y)} \to R\epsilon_*(\Omega^p_{(X_\bullet,Y_\bullet)})$ is a quasi-isomorphism.
\end{thm}

\begin{proof} Consider the commutative diagram
 \[ 
\xymatrix{
    0 \ar[r] & R\epsilon_*\Omega^p_{(X_\bullet,Y_\bullet)} \ar[r] &  R\epsilon_*\Omega^p_{X_\bullet} \ar[r] & 
     R\epsilon_*\Omega^p_{Y_\bullet} \ar[r] & 0 \\
    0 \ar[r] & \tilde{\Omega}^p_{(X_\bullet,Y_\bullet)} \ar[r] \ar[u]^\alpha & \tilde{\Omega}^p_X  \ar[r] \ar[u]^\beta & \tilde{\Omega}^p_Y \ar[r] \ar[u]^\gamma & 0      
} \]
where the top raw is an exact triangle, and the bottom row is a short exact sequence. 
Since $\beta,\gamma$ are quasi-isomorphisms, so is $\alpha$.
\end{proof}

Thus, the filtered complex $(\tilde{\Omega}^*_{(X,Y)},F)$, with $F$ the naive filtration, is a canonical choice 
for the Du Bois complex of the pair $(X,Y)$ (see~\cite{Kov11}).

\begin{cor}
Let $(X,Y)$ be a weakly toroidal pair. Then $(X,Y)$ has Du Bois singularities.
\end{cor}

\begin{proof} We have to show that $\cI_{Y\subset X}=\tilde{\Omega}^0_{(X,Y)}$.
The statement is local, and invariant under \'etale base change. By~\cite{Ar69}, we
may suppose $X=\Spec k[\cM]$ is a weakly normal local model, and $Y\subset X$ is a 
torus invariant close subvariety. By definition, $A^0(X,Y)=\Gamma(X,\cI_{Y\subset X})$. 
Therefore the claim holds.
\end{proof}

As above, we obtain the Poincar\'e lemma for pairs: $\C_{(X^{an},Y^{an})}\isoto \tilde{\Omega}^*_{(X^{an},Y^{an})}$.
Similarly, we obtain

\begin{thm}\label{E1degpairs}
Let $(X,Y)$ be a weakly toroidal pair, with $X/\C$ proper. Then the spectral sequence
$$
E^{pq}_1=H^q(X,\tilde{\Omega}^p_{(X,Y)})\Longrightarrow \Gr^p_F H^{p+q}(X^{an},Y^{an};\C)
$$
degenerates at $E_1$, and converges to the Hodge filtration on the relative cohomology groups of $(X^{an},Y^{an})$.
\end{thm}

We can also generalize~\cite[Propositions 1.8, 2.8]{Dan78} as follows:

\begin{prop}
Let $(X',Y')$ and $(X,Y)$ be weakly toroidal pairs. Let $f\colon X'\to X$ be a proper surjective morphism
such that $Y'=f^{-1}(Y)$ and $f\colon X'\setminus Y'\to X\setminus Y$ is an isomorphism. Then 
$$
\tilde{\Omega}^p_{(X,Y)} \to Rf_*\tilde{\Omega}^p_{(X',Y')}
$$
is a quasi-isomorphism.
\end{prop}

\begin{proof} Consider the commutative diagram
\[ 
\xymatrix{
         \underline{\Omega}^p_{(X,Y)}   \ar[r]  & Rf_*\underline{\Omega}^p_{(X',Y')}     \\
         \tilde{\Omega}^p_{(X,Y)}   \ar[r]  \ar[u]        &  Rf_*\tilde{\Omega}^p_{(X',Y')}   \ar[u]               
} \]
The vertical arrows are quasi-isomorphisms, by Theorem~\ref{mvtlpr}. The top horizontal arrow
is a quasi-isomorphism by the proof of~\cite[Proposition 4.11]{DB81}. Therefore the bottom horizontal 
arrow is a quasi-isomorphism as well.
\end{proof}



\begin{thebibliography}{BLR06}


\bibitem{Am17}
Ambro, F.,
{\em On toric face rings II.}
{preprint 2016.}

\bibitem{AB69} 
Andreotti A., Bombieri E., 
{\em Sugli omeomorfismi delle variet\`{a} algebriche}, 
{Ann. Scuola Norm. Sup. Pisa (3), {\bf 23} (1969), 431--450.} 

\bibitem{Ar69}
Artin, M.,
{\em Algebraic approximation of structures over complete local rings.}
{Inst. Hautes \'Etudes Sci. Publ. Math. No. 36 1969, 23 -- 58.}

\bibitem{Del71}
Deligne P., 
{\em Th\'eorie de Hodge II,}
{Publ. Math. IHES , 40 (1971), 5Ð-57.}

\bibitem{Del74}
Deligne P., 
{\em Th\'eorie de Hodge III,}
{Publ. Math. IHES, 44 (1974), 5--77.}

\bibitem{Sta87}
Stanley, R.P.,
{\em Generalized $h$-vectors, intersection cohomology of toric varieties, and related results.}
{Commutative Algebra and Combinatorics, 
Adv. Stud. Pure Math. {\bf 11} (1987), 187 -- 213.}

\bibitem{BLR06}
Bruns, W.; Li, P.; R\"omer, T.,
{\em On seminormal monoid rings,}
{J. Algebra {\bf 302} (1) (2006), 361--386.}

\bibitem{Dan78}
Danilov, V. I.,
{\em The geometry of toric varieties.}
{Uspekhi Mat. Nauk 33 (2) (1978), 85 -- 134.}

\bibitem{Dan91}
Danilov, V. I.,
{\em De Rham complex on toroidal variety,}
{in LNM {\bf 1479} (1991), 26 -- 38.}

\bibitem{DB81}
Du Bois, P.,
{\em Complexe de de Rham filtr\'e d'une vari\'et\'e singuli\'ere,}
{Bull. Soc. math. France, {\bf 109} (1981), 41 -- 81.}

\bibitem{HR76}
Hochster, M; Roberts, J. L.,
{\em The purity of the Frobenius and local cohomology,} 
{Adv. Math. {\bf 21} (1976) 117 -- 172.}

\bibitem{HJ14}
Huber, A.; J\"order, C.,
{\em Differential forms in the $h$-topology.}
{Alg. Geom. {\bf 1}(4) (2014), 449 -- 478.}

\bibitem{IR07}
Ichim, B.; R\"omer, T.,
{\em On toric face rings.}
{Journal of Pure and Applied Algebra 210 (2007), 249 -- 266.}

\bibitem{Ish80}
Ishida, M.-N.,
{\em Torus embeddings and dualizing complexes.}
{Tohoku Math. J., {\bf 32} (1980), 111--146.}

\bibitem{Ish85}
Ishida, M.-N.,
{\em Torus embeddings and de Rham complexes.}
{Commutative algebra and combinatorics (Kyoto, 1985), 111--145,
Adv. Stud. Pure Math., 11, North-Holland, Amsterdam, 1987.}

\bibitem{KKMS73}
Kempf, G.; Knudsen, F. F.; Mumford, D.; Saint-Donat, B.,
{\em Toroidal embeddings. I.} 
{Lecture Notes in Mathematics, Vol. 339. Springer-Verlag, Berlin-New York, 1973.} 

\bibitem{Kbook}
Koll\'ar, J.,
{\em Singularities of the Minimal Model Program.}
{Cambridge Tracts in Mathematics 200 (2013).}

\bibitem{Kov11}
Kov\'acs, S.,
{\em Du Bois pairs and vanishing theorems.}
{Kyoto J. Math. {\bf 51}(1) (2011), 47 -- 69.}

\bibitem{CLS11}
Cox, D.; Little, J.; Schenck, H.,
{Toric varieties,}
{Graduate Studies in Mathematics {\bf 124} (2011).}

\bibitem{Ngu12}
Nguyen, D.H.,
{\em Homological and combinatorial properties of toric face rings.}
{PhD. Thesis Fachbereich Mathematik/Informatik der Universit\"at Osnabr\"uck (2012).}

\bibitem{Stee77}
Steenbrink, J.H.M,
{\em Mixed Hodge structure on the vanishing cohomology,} 
{in: Real and complex singularities, Sijthoff and Noordhoff, Alphen aan den Rijn (1977),
525 Ð- 563.}

\bibitem{Tr70} 
Traverso C., 
{\em Seminormality and Picard group}, 
{Ann. Sc. Norm. Sup. Pisa {\bf 24} (1970), 585--595.}


\end{thebibliography}
\end{document}